\documentclass{article}
\usepackage[T1]{fontenc}
\usepackage{graphicx}
\usepackage[english]{babel}
\usepackage{amsthm}
\usepackage{amsfonts}
\usepackage{amsmath,amssymb}
\usepackage{color}
\usepackage{authblk}
\usepackage{graphicx}
\usepackage{tikz,pgfplots}
\usepackage{enumerate}
\usepackage[all]{xypic}

\newtheorem{theorem}{Theorem}[section]
\newtheorem{lemma}[theorem]{Lemma}
\newtheorem{proposition}[theorem]{Proposition}
\newtheorem{corollary}[theorem]{Corollary}

\newtheorem{definition}[theorem]{Definition}
\newtheorem{example}[theorem]{Example}

\newtheorem{remark}[theorem]{Remark} 

\newcommand{\op}[1]{\textrm{\upshape #1}}

\newcommand{\join}{\vee}

\newcommand{\meet}{\wedge}

\newcommand{\alg}[1]{{\textbf{\upshape #1}}}  %
\newcommand{\vv}[1]{\mathsf {#1}}

\renewcommand{\th}{\theta}

\newcommand{\sse}{\subseteq}

\newcommand{\vuc}[2]{#1_1,\dots,#1_{#2}}

\newcommand{\imp}{\rightarrow}

\newcommand{\rad}[1]{{\rm{Rad}({\textbf{#1}})}}

\newcommand{\cc}[1]{\mathcal{#1}}
\title{Maximal theories of product logic}
\author{Valeria Giustarini and Sara Ugolini\\
valeria.giustarini@gmail.com;\\ saraugolini.phd@gmail.com}
\date{}
\begin{document}
\maketitle             
\begin{abstract}
Product logic is one of the main fuzzy logics arising from a continuous t-norm, and its equivalent algebraic semantics is the variety of product algebras. In this contribution, we study maximal filters of product algebras, and their relation with product hoops. The latter constitute the variety of $0$-free subreducts of product algebras. Given any product hoop, we construct a product algebra of which the product hoop is (isomorphic to) a maximal filter. This entails that product hoops coincide exactly with the maximal filters of product algebras, seen as residuated lattices. In this sense, we characterize the equational theory of maximal filters of product algebras.
\end{abstract}
\section{Introduction}
Whenever a logic $\cc L$ has a variety $\vv V$ as its equivalent algebraic semantics \`a la Blok-Pigozzi \cite{BP89}, the theories of the logic $\cc L$ correspond to the congruence filters of the free algebras of $\vv V$. In the same flavor, maximally consistent theories correspond to maximal congruence filters of free algebras. This contribution is about the equational theory of the maximal filters in varieties of algebras related to fuzzy logics.

More precisely, we are interested in one of the most relevant axiomatic extensions of H\'ajek Basic Logic $\mathcal{BL}$ \cite{Ha98}, product logic. The latter has been introduced in \cite{HGE1996}, and it is the fuzzy logic associated to the product t-norm (the binary operation of product among real numbers in the real unit interval $[0,1]$). Basic Logic is the logic of continuous t-norms \cite{CEGT00}, and product logic is, together with \L ukasiewicz logic and G\"odel logic, one of the fundamental fuzzy logics in H\'ajek's framework. Indeed, Mostert-Shields' Theorem \cite{KMP00} shows that a t-norm is continuous if and only if it can be built from \L ukasiewicz, G\"odel and product t-norms by the construction of ordinal sum.

Basic logic and its extensions are all algebraizable, and their equivalent algebraic semantics are varieties of BL-algebras. The latter can be seen as particular bounded commutative residuated lattices, in the signature $(\cdot, \to, \land, \lor, 0, 1)$. Congruence filters of BL-algebras (called just {\em filters} in what follows) are subsets of the domain of the algebras,  closed under product and upwards (with respect to the lattice order). 

In this work, we will characterize the equational theory of the maximal filters of product algebras, seen as residuated lattices.
It is indeed easy to see that congruence filters are actually substructures in the $0$-free signature of commutative residuated lattices $(\cdot, \to, \land, \lor, 1)$. As a consequence, all filters (and thus in particular maximal filters) of any variety $\vv V$ of BL-algebras belong, in this sense, to the variety $\vv V_0$ of $0$-free subreducts of algebras in $\vv V$. 

It is then natural to ask whether the converse is true, that is, is any algebra in $\vv V_0$ isomorphic to a maximal filter of some algebra in $\vv V$? Notice that in general, a $0$-free subreduct of a bounded commutative residuated lattice $\alg A$ is not necessarily closed upwards, thus it might not be a filter of $\alg A$. 
However, the question can be answered positively if $\vv V$ is the variety of MV-algebras, the equivalent algebraic semantics of infinite-valued \L ukasiewicz logic. In \cite{ACV10}, the authors indeed start from a Wajsberg hoop (i.e., a 0-free subreduct of some MV-algebra), and construct an MV-algebra of which the Wajsberg hoop is a maximal filter. The same is true if $\vv V$ is the variety of G\"odel algebras, the equivalent algebraic semantics of G\"odel logic. Subreducts of G\"odel algebras are called G\"odel hoops, and if one adds a bottom element to any G\"odel hoop $\alg G$ (extending all operations in the obvious way), this generates a G\"odel algebra of which the G\"odel hoop is the only maximal filter, as we will see in detail in the Preliminaries section. 

We will show the analogous result for the third most relevant fuzzy logic belonging to H\'ajek's framework, i.e., product logic. In particular, taken any product hoop $\alg S$ (that is, a $0$-free subreduct of a product algebra), we construct a product algebra of which $\alg S$ is (isomorphic to) a maximal filter.
\section{Preliminaries}
In this section we introduce the algebraic structures that will be object of our study.
For all the unexplained notions of universal algebra we refer to \cite{BurrisSank}, and for the theory of residuated lattices to \cite{GJKO}.
A {\em bounded commutative integral residuated lattice} (or BCIRL) is an algebra $\alg A = (A, \cdot, \to, \land, \lor, 0, 1)$ of type $(2, 2, 2, 2, 0, 0)$ such that:
\begin{enumerate}
\item $(A, \cdot, 1)$ is a commutative monoid;
\item $(A, \land, \lor, 0, 1)$ is a bounded lattice with $0 \leq x \leq 1$ for all $x \in A$;
\item the residuation law holds: $x \cdot y \leq z$ iff $y \leq x \to z$.
\end{enumerate}
We will often write $xy$ for $x \cdot y$, and $x^n$ for $x \cdot \ldots \cdot x$ ($n$ times). Moreover, we will consider a negation operator defined as $\neg x = x \to 0$.\\
A {\em BL-algebra} is a BCIRL that further satisfies {\em divisibility}:
\begin{equation}\tag{div}
	x \meet y=x(x \to y)
\end{equation}
and {\em prelinearity}:
\begin{equation}\tag{prel}
	(x \to y) \join (y \to x)=1.
\end{equation}
Satisfying the divisibility equation is equivalent to saying that the order $\leq$ induced by the lattice operations is the inverse divisibility ordering, that is, $x \leq y$ if and only if there exists $z$ such that $x = y z$. The prelinearity equation instead characterizes BCIRLs generated by {\em chains}, that is, totally ordered algebras (see \cite{GJKO}). 

In BL-algebras, the lattice operations can actually be rewritten in terms of the monoidal operation and the implication, as:
\begin{eqnarray*}
	x \land y &=& x(x \to y) ,\\
x \join y &=&((x \to y) \to y) \meet ((y \to x) \to x).
\end{eqnarray*}
 Thus we may consider BL-algebras in the language of {\em bounded hoops} (see \cite{BF00,AFM07} for the theory of hoops), that is, as algebras in the signature $(\cdot, \to, 0, 1)$.\\
{\em MV-algebras}, the equivalent algebraic semantics of infinite-valued \L ukasiewicz logic, are BL-algebras satisfying {\em involutivity}:
\begin{equation}
	\neg \neg x=x;
\end{equation} {\em G\"odel algebras}, the equivalent algebraic semantics of G\"odel logic, are BL-algebras satisfying {\em idempotency}:
 \begin{equation}
	x^2 = x;
\end{equation}
{\em product algebras}, the equivalent algebraic semantics of product logic, are BL-algebras satisfying the following identity:
\begin{equation}
\neg x \join ((x \to (x \cdot y))\to y)=1.	
\end{equation}
 Given a variety of BL-algebras, the class of its $0$-free subreducts is a variety of {\em basic hoops} \cite{AFM07}. 
$0$-free subreducts of MV-algebras, G\"odel algebras, and product algebras constitute, respectively, the varieties of Wajsberg, G\"odel, and product hoops.

With respect to their structure theory, all these algebras are very well behaved. In particular, congruences are totally determined by their $1$-blocks (i.e., the set of elements in relation with $1$). If $\alg A$ is a BL-algebra (or a basic hoop), the $1$-block of a congruence of $\alg A$ is called a {\em congruence filter} (or {\em filter} for short). Filters corresponds to the {\em deductive filters} induced by the corresponding logic, which, in the algebra on formulas, are exactly deductively closed theories. It can be shown that a filter of $\alg A$ is a nonempty subset of $\alg A$ closed under multiplication and upwards. It is then easy to prove that a filter $F$ of a BL-algebra (or a basic hoop) $\alg A$, endowed with the inherited operations of $\alg A$, is itself a basic hoop $\alg F$. 

Filters form an algebraic lattice isomorphic with the congruence lattice of $\alg A$ and if $X \sse A$ then the filter generated by $X$ is
$$
\op{Fil}_\alg A(X) = \{a \in A : x_1\cdot \ldots \cdot x_n \le a, \text{for some $n \in \mathbb N$ and $\vuc xn \in X$}\}.
$$
The isomorphism between the filter lattice and the congruence lattice is given by the maps:
\begin{align*}
& \th \longmapsto 1/\th\\
& F \longmapsto \th_F = \{(a,b): a \imp b,\ b \imp a \in F\},
\end{align*}
where $\th$ is a congruence and $F$ a filter. 

We call a filter $F$ \emph{maximal} if it is not contained in any proper filter. Moreover, we call \emph{radical} of an algebra $\alg A$ the intersection of its maximal filters, and we denote it as $\rad A$.
Finally, given a BL-algebra $\alg A$, another subset of its domain that will be relevant in what follows is the {\em Boolean skeleton} of $\alg A$, $\cc B(\alg A)$.
$\cc B(\alg A)$ is the largest Boolean subalgebra of $\alg A$, and its domain is the set of its complemented elements, i.e., elements $x \in A$ such that $x \land \neg x = 0$, $x \lor \neg x = 1$.  

\subsection{Product algebras and product triples}
Product algebras are a variety generated by chains, since they satisfy the prelinearity equation. Product chains can be easily constructed starting with totally ordered {\em cancellative hoops}. The latter are the variety of basic hoops where the monoidal operation is cancellative in the usual sense.
Product chains can be obtained by cancellative hoops with the following construction.

Consider a CIRL $\alg A$, that is, the $0$-free subreduct of some BCIRL. We define its {\em lifting} to be the algebra $\alg 2 \oplus \alg H$, with domain $H \cup \{0\}$, and the operations extending the ones of $\alg H$ in the obvious way: for $x \in H$, $x 0 = 0 x = 0, x \to 0 = 0, 0 \to x = 1$. See Figure \ref{fig:2+H} for a pictorial intuition.

\begin{figure}
\begin{center}
\begin{tikzpicture}
\draw (2,2)  -- (3.5,0);
\draw (0.5,0)-- (2,2);
\filldraw (2,2) circle (1pt);
\node (o) at (2,2.3){$1$};
\draw[dotted, line width=1pt] (0.5,0) -- (3.5,0);
\node at (2,0.8) {$\alg H$};
\filldraw (2,-1) circle (1pt);
\node at (2,-1.3) {$0$};
\end{tikzpicture}
\caption{The algebra $\alg 2 \oplus \alg H$, given any CIRL $\alg H$.}
\label{fig:2+H}
\end{center}
\end{figure}
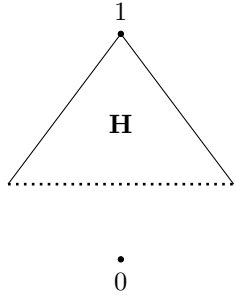

Then, product chains are all the algebras of the form $\alg 2 \oplus \alg C$, for $\alg C$ a totally ordered cancellative hoop \cite{CignoliTorrens00}. Product chains are exactly the finitely subdirectly irreducible product algebras. Subdirectly irreducible product algebras are then the totally ordered algebras of the kind $\alg 2 \oplus \alg C$ with $\alg C$ a subdirectly irreducible cancellative hoop.

This decomposition yields a decomposition of the elements of a product algebra in a {\em Boolean} and a {\em cancellative} component, as shown in \cite{MontagnaUgolini2015} (and in higher generality in \cite{AFU18}).
\begin{lemma}\cite{MontagnaUgolini2015,AFU18}
	Let $\alg P$ be a product algebra. Then every element $x \in \alg P$ can be written as $x = \neg\neg x \cdot (x \lor \neg x)$, where $\neg\neg x \in \cc B(\alg P)$ and $x \lor \neg x \in \rad{\alg P}$.
\end{lemma}
Interestingly, both elements $\neg\neg x$ and $x \lor \neg x$ can be written in the $0$-free language.
\begin{lemma}
	Let $\alg P$ be a product algebra. Then $\neg\neg x = (x \to x^2) \to x$ and $x \lor \neg x = x \to x^2$.
\end{lemma}
\begin{proof}
	The two identities can be easily shown to hold in chains. Indeed, let $x \in \alg 2 \oplus \alg C$, with $\alg C$ totally ordered. Suppose first $x \in C$. Thus $\neg \neg x = 1$, and $(x \to x^2) \to x = x \to x = 1$. Similarly, $x \lor \neg x = x = x \to x^2$. Suppose now $x = 0$, then $\neg\neg 0 = 0 = (0 \to 0^2) \to 0$ and $0 \lor \neg 0 = 1 = 0 \to 0^2$. 
	
	Now, if an equation holds in all the subdirectly irreducible algebras, it holds in every algebra of the variety, thus the claim follows. \qed
\end{proof}
Given the previous lemma, we define for each $x \in \alg P$:
\begin{equation}
	b(x) = (x \to x^2) \to x, \qquad c(x) = x \to x^2.
\end{equation}
Therefore:
\begin{proposition}\label{eq:elementsproductalgebra}
Let $\alg P$ be a product algebra. For each $x \in \alg P$, $x = b(x) \cdot c(x)$, or equivalently $x = b(x) \land c(x)$.
\end{proposition}
The representation of the elements seems to suggest that a product algebra is identified by its Boolean skeleton, and the set of its cancellative elements, but it turns out that they are not enough to identify a unique product algebra (see \cite{MontagnaUgolini2015}).
The algebraic category of product algebras has been indeed shown in \cite{MontagnaUgolini2015} to be equivalent to a category whose objects are triples of the kind $(\alg B, \alg C, \lor_e)$, such that $\alg B$ is a Boolean algebra, $\alg C$ is a cancellative hoop such that $B \cap C = \{1\}$, and $\lor_e$ is a binary operation $\lor_e: \alg B\times \alg C\to \alg C$ intuitively representing the join operation between Boolean and cancellative elements. 
More precisely, for $b \in  B$ and $c \in C$, let 
\begin{equation}
	h_b(x) = b \lor_e x; \qquad k_c(x) = x \lor_e c.
\end{equation}
Then we have the following definition.
\begin{definition}\label{def:extjoin}
	A map $\lor_e: B \times C \to C$ is an {\em external join between $\alg B$ and $\alg C$} if it satisfies the following:
\begin{enumerate}
	\item[(V1)] For fixed $b \in  B$ and $c \in C$, $h_b$ is an endomorphism of $\alg C$ and $k_c$ is a lattice homomorphism from (the lattice reduct of) $\alg B$ into (the lattice reduct of) $\alg C$.
	\item[(V2)] $h_0$ is the identity on $\alg C$, and $h_1$ is constantly equal to 1.
	\item[(V3)] For all $b,b' \in  B$ and for all $c,c' \in  C$, $h_b(c) \lor h_{b'}(c') = h_{b \lor b'} (c \lor c') = h_b(h_{b'} (c \lor c'))$.
	\item[(V4)] For all $b \in B$ and for all $c,c' \in C$, $(b \lor_e c) \cdot c' = (\neg b \lor_e c') \land (b \lor_e (c\cdot c'))$.
\end{enumerate}
\end{definition} 
Given any such a triple $(\alg B, \alg C, \lor_e)$, following \cite{MontagnaUgolini2015}, one can obtain a product algebra of which $\alg B$ is a Boolean skeleton, and $\alg C$ is the radical. First, consider the direct product $B \times C$, and the following equivalence relation $\sim$ on $B \times C$: given $(b,c), (b',c') \in B \times C$, let
\begin{equation}\label{eq:sim}
	(b,c)\sim (b',c') \mbox{ if and only if } b=b' \mbox{ and } \neg b \join_e c=\neg b \join_e c'.
\end{equation}
Let us denote the elements of an equivalence class $(b,c)/\!\!\!\sim$ by $[b,c]$.
Then $\alg B\otimes_{\lor_e} \alg C$ is the product algebra with domain $(B\times C/\!\sim)$, and operations denoted by $(\otimes, \Rightarrow, \sqcap, \sqcup, [0,1], [1,1])$ and defined as follows:
\begin{align*}
&[b,c]\otimes [b',c']=[b \meet b',c \cdot c']\\
&[b,c] \sqcap[b',c']=[b \meet b',c \meet c']\\
&[b,c]\sqcup [b',c']=[b\join b', ((\neg b \join \neg b')\join_e (c \join c')\meet ((b \join \neg b')\join_e c')\meet ((\neg b \join b')\join_e c) ]\\
&[b,c]\Rightarrow [b',c']=[b \to b', \neg b \join_e (c \to c')].
\end{align*}
In \cite{MontagnaUgolini2015} it is shown in particular that every product algebra $\alg A$ is isomorphic to the one obtained by the triple $(\cc B (\alg A), \rad{\alg A}, \lor)$, where $\lor$ is the join of $\alg A$. 

We observe that, given any Boolean algebra $\alg B$, and any cancellative hoop $\alg C$, it is always possible to define an external join between them. Let indeed $M$ be a maximal filter of $\alg B$, we define $\lor_M: B \times C \to C$ as follows:
\begin{equation*}
	b \lor_M c = \left\{ 
	\begin{array}{lr}
	1 & \mbox{ if } b \in M,\\
	c & \mbox{otherwise}.
	\end{array}
	\right.
\end{equation*}
\begin{lemma}\label{lemma:extjoin}
	Let $\alg B$ be a Boolean algebra, $\alg C$ a cancellative hoop, with $\alg B \cap \alg C = \{1\}$. Then $\lor_M: B \times C \to C$ defined above is an external join, and $(\alg B, \alg C, \lor_M)$ is a product triple.
\end{lemma}
\begin{proof}
	Given $b \in B, c \in C$, let $h_b(x)=b \join_M x$ and $k_c(x)=x \join_M c$. We only need to show the properties $(V1) -(V4)$ in Definition \ref{def:extjoin}.
	
For $(V1)$, fix an element $b \in B$, then it is easy to see that $h_b$ is an endomorphism of $\alg C$. Indeed, either $b \in M$, and then $h_b$ is the map constantly equal to $1$, or $b \notin M$, and then $h_b$ is the identity map.
Let now fix some $c \in C$, we need to prove that $k_c$ is a lattice homomorphism from the lattice reduct of $\alg B$ to the lattice reduct of $\alg C$. That is, we need to show that $k_c(x \meet y)=k_c(x) \meet k_c(y)$ and $k_c(x \join y)=k_c(x) \join k_c(y)$.
We show the case of $\land$.
If both $x,y$ are in $M$, then $x \meet y \in M$ as well, thus $k_c(x \meet y)= 1 = k_c(x) \meet k_c(y)$. Otherwise, if at least one element among $x$ and $y$ is not in $M$, $x \land y$ is not in $M$ (since filters are closed upwards). Thus $k_c(x \meet y)= c = k_c(x) \meet k_c(y)$. The case of $\lor$ can be shown analogously. 

For (V2), it follows from the definition that $h_0$ is exactly the identity on $\alg C$, and $h_1$ is constantly equal to $1$. 

We now show (V3), that is, for all $b,b' \in B$ and $c,c' \in C$:
$$h_b(c) \join h_{b'}(c')=h_{b \join b'}(c \join c')=h_b(h_{b'}(c \join c')).$$
If one among $b, b'$ is in $M$, then clearly $h_b(c) \join h_{b'}(c')=h_{b \join b'}(c \join c')=h_b(h_{b'}(c \join c')) = 1$. If instead both $b,b' \not \in M$, then the three terms are all $c \join c'$, since the complement of a maximal Boolean filter is closed under join.

We are left to prove (V4), that is, for all $b \in B$ and $c,c' \in C$, 
$$(b \join_M c) \cdot c'=(\neg b \join_M c') \meet (b \join_M (c \cdot c')).$$
If $b \in M$, then $(b \join_M c) \cdot c'=1 \cdot c'$ and $(\neg b \join_M c') \meet (b \join_M (c \cdot c'))=c'\meet 1=c'$, and the equality holds.
If $b \notin M$, then $(b \join_M c) \cdot c'=c \cdot c'$ and $(\neg b \join_M c') \meet (b \join_M (c \cdot c'))=1 \meet c \cdot c'=c\cdot c'$. This completes the proof.\qed
\end{proof}
\subsection{Maximal filters of Boolean, MV, and G\"odel algebras}
We observe that, in general, a $0$-free subreduct $\alg A$ of a BCIRL $\alg B$ is not necessarily closed upwards, thus it might not be a filter of $\alg A$. Even more, $\alg A$ is not necessarily a filter of the subalgebra $\alg S_{\alg A}$ of $\alg B$ generated by $\alg A$. Indeed, for instance, given any element $x \in A$, its double negation will be in $\alg S_{\alg A}$ and in any filter containing $x$, since $x \leq \neg\neg x$, but $\neg\neg x$ does not have to belong to $\alg A$ (see \cite[Example 2.6]{AglianoUgolini23} for a specific example). 

In this section we will see some known constructions that start from a subreduct $\alg A$ of an algebra $\alg B$ in a variety $\vv V$ of BCIRLs, and obtain a new algebra $\alg C$ in $\vv V$ of which $\alg A$ is isomorphic to a maximal filter. In particular, we will see this for $\vv V$ being the variety of MV, Boolean, and G\"odel algebras. It will also become apparent that  these constructions cannot be used to show the analogous result for the case of product algebras.

Let us start from MV-algebras. In \cite{ACV10}, the authors show that, given any Wajsberg hoop $\alg A$, one can construct an MV-algebra of which $\alg A$ is (isomorphic to) a maximal filter. We recall the construction since we will use it in what follows.
Let $\alg A$ be a Wajsberg hoop, then its {\em MV-closure} is the algebra
$$
{\bf MV}(\alg A)=(A \times \{0,1\}, \cdot_{mv}, \to_{mv}, 0_{mv}, 1_{mv})
$$
where
\begin{align*}
0_{mv}=(1,0), \qquad 1_{mv}=(1,1)
\end{align*}
and, letting $a \oplus b = (a \to ab) \to b$ for $a, b \in A$,
\begin{displaymath} (a,i) \cdot_{mv} (b,j)
\left\{
\begin{array}{l}
(a \cdot b, 1) \mbox{ if } i=j=1\\
(a \to b,0) \mbox{ if } i=1 \mbox{ and } j=0\\
(b \to a,0) \mbox{ if } i=0 \mbox{ and } j=1\\
(a \oplus b,0) \mbox{ if } i=j=0
\end{array}
\right.
\end{displaymath}

\begin{displaymath} (a,i) \to_{mv} (b,j)
\left\{
\begin{array}{l}
(a \to , 1) \mbox{ if } i=j=1\\
(a \cdot b,0) \mbox{ if } i=1 \mbox{ and } j=0\\
(a \oplus b,1) \mbox{ if } i=0 \mbox{ and } j=1\\
(b \to a,1) \mbox{ if } i=j=0.
\end{array}
\right.
\end{displaymath}
Negation is then defined as: $\neg_{mv}(x, i) = (x, i - 1)$.
We notice that ${\bf MV}(\alg A)$ is the disjoint union of the sets $\{(a, 1) : a \in A\}$ and $\{\neg_{mv} (a, 1): a \in A\}$. 

Boolean algebras can be seen as particular MV-algebras such that $x \lor \neg x = 1$ holds. Their $0$-free subreducts constitute the variety of {\em generalized Boolean algebras}, and they have been studied in \cite{Galatos05}. Generalized Boolean algebras can be characterized as idempotent Wajsberg hoops, and every generalized Boolean algebra $\alg G$ is a maximal filter of some Boolean algebra $\alg B$. In particular, let us show that to construct such a Boolean algebra, one can use the MV-closure contruction.
\begin{proposition}\label{prop:booleanclosure}
	Let $A$ be a generalized Boolean algebra, then its MV-closure ${\bf MV}(\alg A)$ is a Boolean algebra.
\end{proposition}
\begin{proof}
	Let $\alg A$ be a generalized Boolean algebra, that is, an idempotent Wajsberg hoop.
We can then apply the MV-closure construction recalled above and obtain an MV-algebra. Since Boolean algebras can be characterized as idempotent MV-algebras, it suffices to check that $\cdot_{mv}$ is idempotent. By definition, 
$$
(x,1)\cdot (x,1)=(x^2,1)=(x,1),
$$
and 
$$
(x,0)\cdot (x,0)=((x \to x^2) \to x, 0)=((x\to x) \to x,0)=(x,0).
$$
Therefore, ${\bf MV}(\alg A)$ is a Boolean algebra.\qed
\end{proof}
We observe that the MV-closure construction cannot be used (as is) in order to construct a product algebra from a product hoop, nor a G\"odel algebra from a G\"odel hoop, since it generates an involutive structure, and the only involutive G\"odel and product algebras are Boolean algebras.

However, given a G\"odel hoop $\alg G$, the lifting $\alg 2 \oplus \alg G$ is a G\"odel algebra of which $\alg G$ is the unique maximal filter, and actually every directly indecomposable G\"odel algebra has this shape (see for instance \cite{AFU18}). 

Finally, we observe that the lifting construction cannot be used for product algebras and MV-algebras. Indeed, in general, given $\alg H$ a product hoop (or a Wajsberg hoop), $\alg 2 \oplus \alg H$ is not necessarily a product algebra (or an MV-algebra). Indeed, we have the following easy counterexample.
\begin{example}
	 Consider $\alg 2_0$ to be the $0$-free reduct of the $\alg 2$ element Boolean algebra. Then $\alg 2_0$ is a generalized Boolean algebra that is both a product hoop and a Wajsberg hoop. Now, $\alg 2 \oplus \alg 2_0$ is the three-element G\"odel chain, that is not a product algebra nor an MV-algebra.
\end{example}

\section{Constructing a product algebra from a product hoop}
In this section, given a product hoop $\alg S$, we construct a product algebra $\alg P(\alg S)$ such that $\alg S$ is (isomorphic to) a maximal filter of $\alg P(\alg S)$. 

Notice that, given a product hoop $\alg S$, we have that $\alg S$ is the $0$-free subreduct of some product algebra $\alg A$. Thus, the elements in $\alg S$ can be represented as in Proposition \ref{eq:elementsproductalgebra}, $x = b(x) \land c(x)$.
We therefore consider the following sets:
\begin{equation}
	\alg C(\alg S) = \{c(x) : x \in S\}, \qquad \alg G(\alg S) = \{b(x) : x \in S\}.
\end{equation}
\begin{lemma}
	Let $\alg S$ be a product hoop, and consider $\alg C(\alg S)$ and $\alg G(\alg S)$ defined as above. Then $\alg C(\alg S)$ is a cancellative hoop and $\alg G(\alg S)$ is a generalized Boolean algebra.
\end{lemma}
\begin{proof}
We assume that $\alg S$ is a subreduct of a product algebra $\alg A$. 

First we show that $\alg C(\alg S)$ is a cancellative hoop. Given any $x \in S$, since $c(x)=x \to x^2$ and $\alg S$ is closed under the hoop operations, we get that $c(x)$ is itself an element of $\alg S$. It follows that $\alg C(\alg S)$ is closed under all the hoop operations as well. Since $\alg C(\alg S)$ is a subset of the radical of $\alg A$, which is a cancellative hoop, it follows that $\alg C(\alg S)$ is a cancellative hoop itself. 

We now consider $\alg G(\alg S)$, and proceed with the analogous reasoning as above. Given any $x \in S$, $b(x) = (x \to x^2) \to x \in S$, thus $\alg G(\alg S)$ is closed under the hoop operations and is a subset of the Boolean skeleton of $\alg A$, that is a Boolean algebra. Therefore, $\alg G(\alg S)$ is the subreduct of a Boolean algebra, that is, a generalized Boolean algebra.\qed
\end{proof}
Let then $\alg B(\alg S) = {\bf MV}(\alg G(\alg S))$ be the MV-closure of $\alg G(\alg S)$.  By Proposition \ref{prop:booleanclosure}, $\alg B(\alg S)$ is a Boolean algebra of which $\alg G(\alg S)$ is (isomorphic to) a maximal filter.
We can then define a binary operation $\lor_S: \alg B(\alg S) \times \alg C(\alg S) \to \alg C(\alg S)$ as follows:
\begin{equation*}
	b \lor_S c = \left\{ 
	\begin{array}{lr}
	1 & \mbox{ if } b \in \alg G(\alg S),\\
	c & \mbox{otherwise}.
	\end{array}
	\right.
\end{equation*}
By Lemma \ref{lemma:extjoin}, we obtain the following.
\begin{proposition}
	$(\alg B(\alg S), \alg C(\alg S), \lor_S)$ is a product triple.
\end{proposition}
Thus, we can consider the product algebra associated to the product triple $(\alg B(\alg S), \alg C(\alg S), \lor_S)$, and define:
\begin{equation}
	\alg P(\alg S) = \alg B(\alg S) \otimes_{\lor_S} \alg C(\alg S).
\end{equation}
	We are now ready to show that, given a product hoop $\alg S$, we can construct a product algebra of which $\alg S$ is a maximal filter. 
\begin{theorem}
Let $\alg S$ be a product hoop. Then $\alg S$ is isomorphic to a maximal filter of $\alg P(\alg S)$.
\end{theorem}
\begin{proof}
Let $x \in S$, then $x = b(x) \land c(x)$. Let $f: \alg S \to \alg P(\alg S)$ be defined by $f(x) = [b(x), c(x)]$. It can be directly checked that $f$ is a hoop isomorphism from $\alg S$ to the subset of $\alg P(\alg S)$ given by the elements $ \alg S(\alg S) = \{(b, c) : b \in \alg G(\alg S), c \in \alg C(\alg S)\}$. For the reader interested in the details, since we can see $\alg S$ as a subreduct of a product algebra $\alg A$, the fact that $f$ is a hoop homomorphism is a consequence of \cite[Theorem 6.1]{MontagnaUgolini2015}. In particular from the part (b) of the proof of \cite[Theorem 6.1]{MontagnaUgolini2015}, we get exactly that $f$ is a hoop isomorphism from $\alg S$ to the elements in $\alg S(\alg S)$. 

We now show that $\alg S(\alg S)$ is a filter of $\alg P(\alg S)$. Since $\alg S(\alg S)$ is a product hoop (because $\alg S$ is), it is closed under product and it contains $[1,1]$. Suppose now $[b(x), c(x)] \in \alg S(\alg S)$ and $[b(x), c(x)] \leq [b(y), c(y)]$, or equivalently, $$[b(x),c(x)]\Rightarrow [b(y),c(y)]= [1,1].$$
 We want to show that $[b(y), c(y)] \in \alg S(\alg S)$. By definition of the operations,
$$[b(x),c(x)]\Rightarrow [b(y),c(y)]=[b(x) \to b(y), \neg b(x) \join_S (c(x) \to c(y))].$$
Thus, by the definition of the equivalence relation $\sim$ in (\ref{eq:sim}), $b(x) \to b(y)=1$, or equivalently, $b(x) \leq b(y)$. Since $\alg G(\alg S)$ is a filter of $\alg B(\alg S)$, $b(y) \in \alg G(\alg S)$. Moreover, again by the definition of $\sim$, we get that:
 $$\neg b(x) \join_S (c(x) \to c(y))=1,$$
 and since  $\neg b(x)$ is not in $\alg G(\alg S)$ (given that $b(x)$ is), we get that $c(x)\to c(y)=1$, or equivalently $c(x) \leq c(y)$. Since $\alg C(\alg S)$ is the radical of $\alg P(\alg S)$, it is a filter, and thus $c(y) \in \alg C(\alg S)$. Therefore, $[b(y), c(y)] \in \alg S(\alg S)$, and $\alg S(\alg S)$ is a filter of $\alg P(\alg S)$.

Now it is left to show that $\alg S(\alg S)$ is a maximal filter, that is, it is not contained in any proper filter. Suppose that $[b(x),c(x)] \in \alg P(\alg S), [b(x),c(x)] \notin \alg S(\alg S)$. Then necessarily $b(x) \not \in \alg G(\alg S)$, meaning that $\neg b(x) \in \alg G(\alg S)$. Thus if we consider a filter $F$ of $\alg P (\alg S)$ which includes both $\alg S(\alg S)$ and $[b(x),c(x)]$, the following product is also in $F$:
$$
[b(x),c(x)]\cdot [\neg b(x),c(x)]=[b(x) \meet \neg b(x), (c(x))^2]=[0,(c(x))^2]=[0,0].
$$
That is, $F$ is not proper. 
We conclude that $\alg S(\alg S)$ is a maximal filter of $\alg P(\alg S)$, and the proof is complete.\qed
\end{proof}
Therefore:
\begin{corollary}
Product hoops coincide with the class of maximal filters of product algebras, seen as residuated lattices with the restricted operations. 	
\end{corollary}
Let us show the following particular instance of our construction.
\begin{example}
Let $\alg S$ be a cancellative hoop. Then $\alg P(\alg S) \cong \alg 2 \oplus \alg S$.
Indeed, if $\alg S$ is a cancellative hoop, $\alg G(\alg S)=\{1\}$. As a consequence, $\alg B(\alg S) \cong \alg 2$. Moreover, $\alg C(\alg S)=\alg S$. Then, we can see the elements of $\alg P(\alg S)$ as either of the form $[1, c]$ or $[0,c]$, for $c \in \alg S$. By the equivalence relation $\sim$ in (\ref{eq:sim}), all of the elements of the form $(0,c)$ belong to the same equivalence class, while $(1, c) \sim (1, c')$ if and only if $c = c'$. Thus, we can see the domain of $\alg P(\alg S)$ as $\{[1, s]: s \in S\} \cup \{[0,1]\}$.  
It follows from the definition of the operations that $\alg P(\alg S)$ is isomorphic to $\alg 2 \oplus \alg S$.
\end{example}

\begin{example}	
Let $\alg S$ be the $0$-free reduct of a product algebra $\alg P = \alg 2 \oplus\alg C$, where $\alg C$ is a cancellative hoop.  Then, $\alg P(\alg S) \cong \alg 2 \times ( \alg 2 \oplus \alg C)$.
Indeed, $\alg G(\alg S) = \alg 2$, which means that $\alg B(\alg S) \cong \alg 4$, the Boolean algebra of $4$ elements. Moreover, $\alg C(\alg S)=\alg C$. As a result, the elements of $\alg P(\alg S)$ are of the form $[1,c], [0,c],[\neg 0,c]$ or $[\neg 1,c]$, for $c \in C$.  By the equivalence relation $\sim$ in  (\ref{eq:sim}), all the elements of the form $(\neg 0,c)$ belong to the same equivalence class, as well as the ones of the form $(\neg 1,c)$. On the other hand, $(1,c) \sim (1,c')$ if and only if $c=c'$, and $(0,c) \sim (0,c')$ if and only if $c=c'$. Thus, the domain of $\alg P(\alg S)$ is given by: $\{[1,c] : c \in C\} \cup \{[0,c]: c \in C\} \cup \{[\neg 0,1]\} \cup \{[\neg 1,1]\}.$
It can then be directly checked that $\alg P(\alg S)$ is isomorphic to $\alg 2 \otimes (\alg 2 \oplus \alg C)$. If $\alg C$ is a chain, the Hasse diagram of $\alg P(\alg S)$ is as in Figure \ref{1}. 

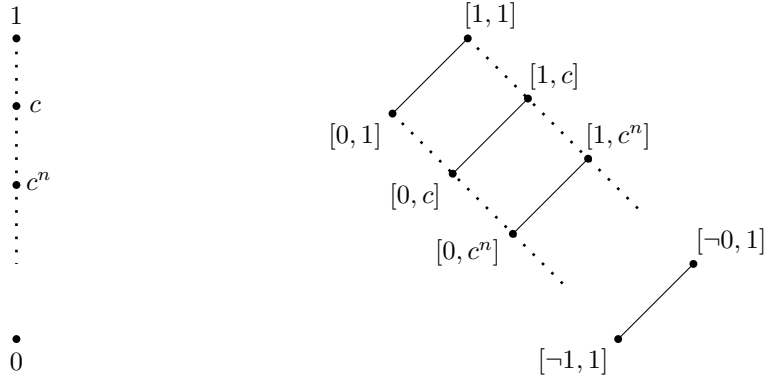
\begin{figure}
\begin{center}
\begin{tikzpicture}
\fill (0, 5) circle (1.5pt);
\fill (0, 1) circle (1.5pt);
\fill (0, 4.1) circle (1.5pt);
\fill (0, 3.05) circle (1.5pt);
\draw [loosely dotted, line width = 1pt] (0,5) -- (0,2);
\node at (0,5.3) {$1$};
\node at (0,0.7) {$0$};
\node at (0.25,4.1) {$c$};
\node at (0.3,3.1) {$c^n$};
\fill (6,5) circle (1.5pt);
\fill (5,4) circle (1.5pt);
\fill (6.8,4.2) circle (1.5pt);
\fill (7.6,3.4) circle (1.5pt);
\fill (5.8,3.2) circle (1.5pt);
\fill (6.6,2.4) circle (1.5pt);
\fill (9,2) circle (1.5pt);
\fill (8,1) circle (1.5pt);
\draw [loosely dotted, line width = 1pt] (5,4) -- (7.3,1.7);
\draw [loosely dotted, line width = 1pt] (6,5) -- (8.3,2.7);
\draw (6,5) -- (5,4);
\draw (5.8,3.2) -- (6.8, 4.2);
\draw (6.6, 2.4) -- (7.6, 3.4);
\draw (8,1) -- (9,2);
\node at (6.3,5.3) {$[1,1]$};
\node at (7.15,4.5) {$[1,c]$};
\node at (8,3.7) {$[1,c^n]$};
\node at (9.5,2.3) {$[\neg 0,1]$};
\node at (4.5,3.7) {$[0,1]$};
\node at (5.3,2.9) {$[0,c]$};
\node at (6,2.15) {$[0,c^n]$};
\node at (7.4,0.7) {$[\neg 1,1]$};
\end{tikzpicture}
\caption{On the left, $\alg 2 \oplus \alg C$, given a cancellative hoop chain $\alg C$. On the right, $\alg P(\alg 2 \oplus \alg C)$.}
\label{1}
\end{center}
\end{figure}
\end{example}
\begin{remark}
	An interesting observation stems from the last example. Given any BCIRL $\alg A$ in a variety $\vv V$, the direct product $\alg 2 \times \alg A$ is in $\vv V$. Moreover, the $0$-free reduct of $\alg A$ is isomorphic to the maximal filter of $\alg 2 \times \alg A$ given by the elements $\{(1, a): a \in A\}$. Thus, for any variety of BCIRLs $\vv V$, and $\alg A \in \vv V$, the $0$-free reduct of $\alg A$ belongs to the class of maximal filters. Notice that this does not mean that any $0$-free {\em subreduct} of $\alg A$ does.
\end{remark} 

\section{Conclusions}
We have shown that, given a product hoop $\alg S$, we can construct a product algebra $\alg P$ such that $\alg S$ is isomorphic to a maximal filter of $\alg P$. Since every maximal filter can be seen as a product hoop, in this sense, we have characterized the equational theory of maximal filters of product algebras.

At present, we do not know if this property is shared by bounded residuated lattices in general. That is, given a variety $\vv V$ of BL-algebras (or, more in general, BCIRLs), does the variety $\vv V_0$ of $0$-free subreducts of algebras in $\vv V$ coincide with the class of maximal filters of algebras in $\vv V$, seen as residuated lattices? We have seen that this holds for the equivalent algebraic semantics of the three most relevant extensions of H\'ajek Basic Logic.  

In future work we plan to extend and generalize our approach to a larger class of residuated structures. In particular, we observe that the triple construction in \cite{MontagnaUgolini2015}, that we have used in this work, has been extended and generalized in \cite{AFU18} and \cite{BMU} to encompass a large class of residuated structures. In particular, this class includes several important varieties related to fuzzy logics and other nonclassical logics, among which: the variety generated by perfect MV-algebras, nilpotent minimum algebras, Stonean Heyting algebras, regular Nelson residuated lattices. We plan to study to what extent our construction can be extended to this wider setting. 

Moreover, the representation of the elements of product algebras naturally gives a representation of the elements of product hoops, hinting at a triple-like representation for such algebras as well. We believe this to be another interesting line of research which we will investigate in future work.


\begin{thebibliography}{8}

\bibitem{ACV10}
Abad, M.,Casta\~{n}o, D., Varela, J.: MV-closures of Wajsberg hoops and applications. Algebra Universalis {\bf 64}, 213--230 (2010).

\bibitem{AFM07}
Aglian\`o, P., Ferreirim, I.M.A., Montagna, F.: Basic hoops: an algebraic study of continuous t-norms. Studia Logica {\bf 87}(1), 73--98 (2007).

\bibitem{AglianoUgolini23}
Aglian\`o, P., Ugolini, S.: Projectivity and unification in substructural logics of generalized rotations. International Journal of Approximate Reasoning {\bf 153}, 172--192 (2023).

\bibitem{AFU18}
Aguzzoli, S., Flaminio, T., Ugolini, S.: Equivalences between subcategories of MTL-algebras via Boolean algebras and prelinear semihoops. Journal of Logic and Computation {\bf 27}(8), 2525--2549 (2017).
\bibitem{BP89}
Blok, W.J., Pigozzi, D.: Algebrizable Logics, Vol.396, 396, Memoirs of the American Mathematical Society, Providence
January (1989).

\bibitem{BF00}
Blok, W.J., Ferreirim, I.M.A.: On the structure of hoops. Algebra Universalis {\bf 43}, 233--257 (2000).

\bibitem{BMU}
Busaniche, M., Marcos, M., Ugolini, S.: Representation by triples of algebras with an MV-retract. Fuzzy Sets and Systems {\bf 369}, 82--102 (2019).

\bibitem{BurrisSank}
Burris, S., Sankappanavar, H. P.: A course in Universal Algebra, Springer- Velag, New York (1981).

\bibitem{CEGT00}
Cignoli, R., Esteva, F., Godo, L., Torrens A.: Basic Fuzzy Logic is the logic of continuous t-norms and their residua. Soft Computing {\bf 4}, 106--112 (2000). 

\bibitem{CignoliTorrens00}
Cignoli, R., Torrens, A.: An algebraic analysis of product logic. Multiple-Valued Logic {\bf 5}, 45--65 (2000).

\bibitem{Galatos05}
Galatos, N.: Minimal varieties of residuated lattices. Algebra Universalis {\bf 52}(2), 215--239 (2005).

\bibitem{GJKO}
Galatos, N., Jipsen, P., Kowalski, T., and Ono, H.: Residuated Lattices: An Algebraic Glimpse at Substructural Logics. Studies in Logics and the Foundations of Mathematics, {\bf 151}, Elsevier, Amsterdam, The Netherlands (2007).

\bibitem{Ha98}
~H{\'a}jek, P.: Metamathematics of fuzzy logic. Trends in Logic-Studia Logica Library (4), Kluwer Academic Publ., Dordrecht, Boston,London (1998).

\bibitem{HGE1996}
~H{\'a}jek,P.,  Godo, L. ,Esteva, F.: A complete many-valued logic with product conjunction. Arch. Math. Logic {\bf 35}, 191--208 (1996).
\bibitem{KMP00}
Klement, E.P., Mesiar, R., Pap,E.: Triangular Norms. Kluwer Academic Publishers, Dordrecht (2000).

\bibitem{MontagnaUgolini2015}
Montagna, F., Ugolini, S.: A categorical equivalence for product algebras. Studia Logica, {\bf 103 }(2), 345--373 (2015).
\end{thebibliography}
\end{document}